\documentclass{IEEEtran}
\usepackage{cite}
\usepackage{amsmath,amssymb,amsfonts}
\usepackage{algorithmic}
\usepackage{algorithm}
\usepackage{graphicx}
\usepackage{textcomp}
\usepackage{amsthm}
\usepackage{multirow}
\usepackage{xcolor}
\usepackage{mathtools}
\usepackage{siunitx}

\usepackage[algo2e,ruled,vlined,linesnumbered,norelsize]{algorithm2e}
\makeatletter
\newcommand{\removelatexerror}{\let\@latex@error\@gobble}
\makeatother

\theoremstyle{definition}
\newtheorem{definition}{Definition}
\newtheorem{theorem}{Theorem}
\newtheorem{lemma}{Lemma}

\theoremstyle{remark}
\newtheorem{remark}{Remark}
\theoremstyle{assumption}

\allowdisplaybreaks[4]
\def\iter{i}
\def\maxiter{\iter_{\mathit{max}}}
\def\lsiter{\iter^\prime}
\def\maxlsiter{\maxiter^\prime}
\def\algtol{\mathit{tolerance}}
\def\infeasflag{\mathit{LineSearch}}
\def\old{\mathrm{old}}
\def\algfalse{\mathit{false}}
\def\algtrue{\mathit{true}}

\def\BibTeX{{\rm B\kern-.05em{\sc i\kern-.025em b}\kern-.08em
    T\kern-.1667em\lower.7ex\hbox{E}\kern-.125emX}}
\begin{document}
\title{Safe adaptive NMPC using ellipsoidal tubes}
\author{Johannes Buerger and Mark Cannon
\thanks{Johannes Buerger is with the BMW Group, Munich, Germany (e-mail:
johannes.buerger@bmw.de). }
\thanks{Mark Cannon is with the Department of Engineering Science, University of Oxford, U.K. (e-mail:
mark.cannon@eng.ox.ac.uk).}}

\maketitle

\begin{abstract}
A computationally efficient nonlinear Model Predictive Control (NMPC) algorithm is proposed for safe learning-based control with a system model represented by an incompletely known affine combination of basis functions and subject to additive set-bounded disturbances. The proposed algorithm employs successive linearization around predicted trajectories and accounts for the uncertain components of future states due to linearization, modelling errors and disturbances using ellipsoidal sets centered on the predicted nominal state trajectory. An ellipsoidal tube-based approach ensures satisfaction of constraints on control variables and model states. Feasibility is ensured using local bounds on linearization errors and a procedure based on a backtracking line search. We combine the approach with a set membership parameter estimation strategy in numerical simulations. We show that the ellipsoidal embedding of the predicted uncertainty scales favourably with the problem size. The resulting algorithm is recursively feasible and provides closed-loop stability and performance guarantees.
\end{abstract}


\section{Introduction}\label{introduction}
Model Predictive Control (MPC) is an optimal control approach with solid theoretical foundations and extensive applications in engineering practice~\cite{Kou16}. 
In many cases, accurate prediction models can be obtained using physical modelling or black-box estimation methods. However, due to the inherent uncertainty in any system model, provably safe robust and adaptive MPC strategies are a key focus of this research field.

There has been significant recent interest in integrating learning-based methods with robust MPC approaches~\cite{Hew20}. Like robust adaptive MPC algorithms, these techniques retain the primary benefits of robust MPC while leveraging information about the controlled system collected during the execution of a control task. Thus they are able to improve the accuracy of the system model and enhance closed-loop control performance.

While several methods for robust adaptive linear MPC have been proposed \cite{Lor19,Lu21}, the more general case of robust adaptive NMPC has received comparatively little attention \cite{Koh21,Ade09,Bue24}.
The approaches of \cite{Koh21} and \cite{Ade09} have strong system-theoretical properties, but are computationally intensive because they require the solution of a nonconvex program online. The method presented in~\cite{Bue24} is both theoretically sound and computationally efficient for low-order systems (due to its foundations in convex optimization). 
However, a key drawback of \cite{Bue24} is its limited scalability, since the number of optimization variables depends on the number of vertices of  polytopic sets describing the predicted state tube cross sections. 

This paper considers an alternative adaptive NMPC approach based on sequential convex approximations, but considering ellipsoidal tubes to bound the effects of the uncertainty in prediction. 
In this work we extend the theory of \cite{Can11} to the context of systems with additive and parametric uncertainty, and propose a robust adaptive nonlinear model predictive control algorithm based on set membership parameter estimation (SME) \cite{Bue24}. The resulting algorithm relies on the solution of a second order cone program, which provides a computationally efficient and scalable approach to solve the otherwise challenging robust NMPC problem, as we demonstrate via extensive numerical simulations.
To ensure recursive feasibility and closed-loop stability, the perturbations around state and control linearization points are limited to regions where the model approximation is meaningful and the effect of the approximation error is bounded by constructing tubes containing the predicted trajectories.
By incorporating a line search procedure into the online algorithm, the approach provides recursive feasibility, constraint satisfaction and performance guarantees, even if just one linearisation iteration is computed at each online time step.

\textit{Notation:} $\mathbb{N}_{\geq 0}$ is the set of non-negative integers, $\mathbb{N}_{[p,q]} = \{n\in\mathbb{N} : {p \leq n \leq q}\}$, and $\mathbb{N}_q =  \mathbb{N}_{[1,q]}$.
%
%
The $i$th row of a matrix $A$ is $[A]_i$.
For a matrix $A$, the inequality $A \geq 0$ applies elementwise, and $A \succeq 0$ (or $A \succ 0$) indicates that $A$ is positive semidefinite (positive definite).  
The Euclidean and infinity norms are $\|x\|$ and $\|x\|_\infty$, and $\|x\|_Q = (x^\top Q x)^{1/2}$.
%
%
%

\section{Problem statement}\label{problem_statement}

We consider a nonlinear system with unknown (but learnable) parametric dependency subject to an additive disturbance (caused by model error and physical disturbance inputs)
\begin{equation}\label{eq:system}
x_{t+1} = f(x_t,u_t,\theta) + w_t ,
\end{equation}
where $(x_t,u_t) \in \mathcal{X}\times\mathcal{U}$, $w_t\in \mathcal{W}$, $\theta \in \Theta_0$ and $t$ is the discrete time index.
The state, control and disturbance input, and model parameters belong to
bounded polytopic sets:
$\mathcal{X} = \{x \in \mathbb{R}^{n_x}  : E x \leq 1\}$,
$\mathcal{U} = \{u \in \mathbb{R}^{n_u} : G u \leq 1\}$,
$\mathcal{W}=\mathrm{co}\{w^{(r)}, \, r \in\mathbb{N}_{\nu_{w}}\}$,
$\Theta_0 = \{\theta \in \mathbb{R}^{n_{\theta}} : H_{\Theta} \theta \leq h_0\}=\mathrm{co}\{\theta^{(q)},\, q \in\mathbb{N}_{\nu_{\theta}}\}$.

We assume that the system can be represented as an affine combination of known basis functions $f_i(x,u)$, $i \in\smash{\mathbb{N}_{n_{\theta}}}$,
\begin{equation}\label{eq:theta_expansion}
f(x_t,u_t,\theta) = f_0(x_t,u_t) + \sum_{i=1}^{\smash{n_{\theta}}} {\theta_i f_i(x_t,u_t)} .
\end{equation}
where $f_i$ is differentiable and Lipschitz continuous on $\mathcal{X}\times\mathcal{U}$ and $f_i(0,0) = 0$ for all $i$. 
Our approach can be used with any parameter learning algorithm that provides a polytopic parameter set $\Theta_t$ satisfying
$\theta \in \Theta_t \subseteq \Theta_{t-1}$ for all $t>0$.

The objective of the control problem is to minimize a quadratic regulation cost (with $R\succ0$, $Q\succeq 0$) defined by
\begin{equation} \label{eq:cost}
\sum_{t=0}^{\infty} (\left\lVert x_t \right\rVert^2_Q +  \left\lVert u_t \right\rVert^2_R).
\end{equation}
To mitigate the effects of model uncertainty on predicted future state and control sequences, we consider the decision variables to be the variations $\{v_t, v_{t+1} ,\ldots\}$ relative to a feedback law,
\[
u_t = K x_t + v_t ,
\]
where the feedback gain $K$ is robustly stabilizing locally around $x=0$, in the sense defined in Section~\ref{sec:termset}.
To simplify notation, for all $(x,v)$ such that $(x,Kx+v)\in\mathcal{X}\times\mathcal{U}$ let
\[
f_{K}(x,v,\theta) = f(x,Kx+v,\theta)
\]
and $f_{K,i}(x,v) = f_i(x,Kx+v)$ for each $i\in\{0,\ldots,p\}$.

\section{Linearization error bounds in prediction}\label{sec:error_bounds}
We consider the Taylor expansion of the model in~\eqref{eq:system} around a nominal trajectory $\mathbf{x}^0 = \{\smash{x^0_0},\ldots, \smash{x^0_{N}}\}$ defined for a given
sequence $\mathbf{v}^0 = \{\smash{v^0_0}, \ldots, \smash{v^0_{N-1}}\}$ 
and parameter $\theta^0\in\Theta$ 
by
\begin{equation}\label{eq:nominal_system}
x^0_{k+1} = f_K(x_k^0, v^0_k, \theta^0), \ k = 0,\ldots, N-1.
\end{equation}
Defining state and control perturbations $s_k$ and $v_k$, where $x_k = x_k^0 + s_k$ and $u_k = Kx_k + v^0_k + v_k$, we have
\begin{equation}\label{eq:taylor}
x^0_{k+1}+s_{k+1} =  f_K(x^0_k,v^0_k,\theta^0) + \delta^0_k + \Phi_k s_k + B_k v_k + \delta^1_k + w_k
\end{equation}
where $\Phi_k = \nabla_x f_K (x^0_k,v^0_k,\theta^0)$ and $B_k = \nabla_v f_K (x^0_k,v^0_k,\theta^0)$ denote the Jacobian matrices of $f_K$ with respect to $x$ and~$v$.

The perturbations on the state and control input are constrained to satisfy bounds $s_k\in \mathcal{S}$, $v_k\in \mathcal{V}$ for all $k\in \mathbb{N}_{[0,N]}$, where $\mathcal{S}$, $\mathcal{V}$ are given polytopic sets containing the origin. The perturbation $s_k$ contains both a nominal component (due to the control perturbation $v_k$) and an uncertain component (due to linearization errors, parameter estimation errors and external disturbances prior to the $k$th time step of the prediction horizon). In this context we distinguish the zero-order error term $\delta^0$ and the first-order error term $\delta^1$ resulting respectively from parameter estimation and linearization errors as follows.

\begin{definition}\label{def:delta} The error terms $\delta^0_k$ and $\delta^1_k$ in \eqref{eq:taylor} are defined by
\begin{align*}
\delta^0_k &= f_K(x_k^0, v^0_k, \theta) - f_K(x_k^0, v^0_k, \theta^0)
\\
\delta^1_k &= f_K(x_k^0 + s_k, v^0_k + v_k, \theta) - f_K(x_k^0, v^0_k, \theta) - \Phi_k s_k \!-\! B_k v_k
\end{align*}
\end{definition}


Bounds on $\delta^0_k$ are derived from the affine dependence on $\theta$:
\[
\delta^0_k = f_K(x^0_k, v^0_k,{\theta}) -f_K(x^0_k, v^0_k,\theta^0) = f_K (x^0_k, v^0_k, \theta - \theta^0) ,
\]
which implies a polytopic additive disturbance bound: 
\begin{equation}\label{eq:W0_def}
\delta^0_k \in \mathcal{W}^0_k = \mathcal{W}^0(x^0_k,v^0_k,\theta^0,\Theta) = \mathrm{co}
\{\delta^{0\,(q)}_k, \, q \in \mathbb{N}_{\nu_\theta}\}.
\end{equation}
A tight bound can be obtained by elementwise maximization over $\theta \in \Theta$. This set can be recomputed online (based on the current trajectory and parameter estimation set) or can be computed as an offline, global bound.

Bounds on $\delta^1_k$ can similarly be derived using Definition~\ref{def:delta}. The mean value theorem implies, for some $(s,v)\in \mathcal{S}\times \mathcal{V}$,
\begin{align}
\delta^1_k &= \bigl(\nabla_x f_K(x^0_k+s, v^0_k+v, {\theta}) - \Phi_k\bigr) s_k
\nonumber \\
&\quad + \bigl(\nabla_v f_K (x^0_k+s, v^0_k+v, {\theta})- B_k \bigr) v_k . 
\label{eq:mvt_delta1}
\end{align}
A corresponding polytopic uncertainty set can be defined
\begin{align}
\delta^1_k \in \mathcal{W}^1_k 
&= \mathcal{W}^1  (x^0_k,v^0_k,\theta^0, \mathcal{S}, \mathcal{V}, \Theta) \nonumber
\\
&\quad = \mathrm{co} \{ C^{(j)}_ks_k + D^{(j)}_k v_k, \, j \in \mathbb{N}_{\nu_{1}} \}
\label{eq:W1_def}
\end{align}
where $\{C^{(j)}_k,D^{(j)}_k,\, j \in \mathbb{N}_{\nu_{1}}\} $ are determined, for example, by computing componentwise bounds on the Jacobian matrices in~(\ref{eq:mvt_delta1}) over $s\in \mathcal{S}$, $v\in \mathcal{V}$, $\theta \in \Theta$. This set can be recomputed online using the current nominal trajectory and estimated parameter  set or computed offline as a global bounding set.

\section{Tube membership conditions}
We consider the state decomposition $x_k = x^0_k + s_k$ and the control input decomposition $u_k = Kx_k + v^0_k + v_k$, where $s_k$ contains both the nominal and uncertain effects resulting from $w_k$, $\delta^0_k$ and $\delta^1_k$.
The perturbation dynamics can therefore be expressed:
\begin{equation}
s_{k+1} = \Phi_k s_k + B_k v_k + w_k + \delta^0_k + \delta^1_k
\end{equation}
where $w\in \mathcal{W}$, $\delta^0_k\in\mathcal{W}^0_k$ and $\delta^1_k\in\mathcal{W}^1_k$.
We split $s_k$ into nominal and uncertain components, denoted $z_k$ and $e_k$:
\begin{align}
s_k &= z_k + e_k \label{eq:decomp} \\
z_{k+1} &= \Phi_k z_k + B_kv_k \label{eq:z_dynamics}\\
e_{k+1} &= \Phi_k e_k + w_k + \delta^0_k + \delta^1_k
\label{eq:e_dynamics}
\end{align} for $k = 0,\ldots,N-1$.
Following the approach of \cite{Can11} we use ellipsoidal tube cross sections to bound the effects of uncertainty over the prediction horizon:
\begin{equation}\label{eq:tube_cond1}
e_k \in \mathcal{E}(V, \beta_k^2), 
\quad k = 0,\ldots,N 
\end{equation}
where $\mathcal{E}(V, \beta^2)= \{e: e^\top V e\leq \beta^2\}$. The design of ${V\succ0}$ is discussed in Section~\ref{sec:termset}. We note that the shape of the ellipsoidal tube is determined by the matrix $V$ and only the centre and scaling are optimized in the online optimization.

The conditions in~(\ref{eq:tube_cond1}) are ensured recursively via the tube membership conditions on $v_k$, $z_k$ and $\beta_k$,
\begin{multline}\label{eq:tube_mem_cond}
\mathcal{E}(V, \beta_{k+1}^2) \ni \Phi_k e + w + \delta^0 + \delta^1 , \\ 
\forall w \in \mathcal{W}, 
\ \forall \delta^0 \in \mathcal{W}^0_k, 
\ \forall \delta^1 \in \mathcal{W}^1_k, 
\ \forall e\in \mathcal{E}(V, \beta_k^2) .
\end{multline}
A sufficient condition for (\ref{eq:tube_mem_cond}) is given by, for all $j\in\mathbb{N}_{\nu_{1}}$, $q\in\mathbb{N}_{\nu_\theta}$, $r\in\mathbb{N}_{\nu_w}$, and all $e\in\mathcal{E}(V,\beta_k^2)$,
\begin{equation}\label{eq:tube_mem_cond_beta} 
\beta_{k+1} \geq 
\lVert C^{(j)}_k z_k + D^{(j)}_k v_k + \delta^{0\,(q)}_k \rVert_V + \lVert (\Phi_k + C^{(j)}_k)e + w^{(r)}\rVert_V .
\end{equation}
We enforce this condition using the following observation.

\begin{lemma}\label{lem:beta_mode1_dynamics}
Condition (\ref{eq:tube_mem_cond_beta}) holds for all $e\in\mathcal{E}(V,\beta_k^2)$ if
\begin{equation}
\beta_{k+1} \geq (\lambda_k \beta_k^2 + \sigma^2)^{\frac{1}{2}} + \lVert C^{(j)}_k z_k + D^{(j)}_k v_k + \delta^{0\,(q)}_k \rVert_V
\label{eq:beta_mode1_dynamics}
\end{equation}
for all $j\in\mathbb{N}_{\nu_{1}}$, $q\in\mathbb{N}_{\nu_\theta}$, with $\lambda_k$ defined by
\begin{equation}
\lambda_k = \max_{j\in\mathbb{N}_{\nu_1},\, r\in\mathbb{N}_{\nu_r}} \lVert
(\Phi_k + C_k^{(j)})V^{-\frac{1}{2}} \rVert_{\Psi^{(r)}}^2
\label{eq:lambda_mode1_def}
\end{equation}
where $\Psi^{(r)} = (V^{-1} - w^{(r)} w^{(r)\, \top} \sigma^{-2})^{-1}$ and where $\sigma$ is a constant whose design is discussed in Section~\ref{sec:termset}.
\end{lemma}

\begin{proof}
Let $\mu^{(j,q)} = \beta_{k+1} - \lVert C^{(j)}_k z_k + D^{(j)}_k v_k + \delta^{0\,(q)}_k \rVert_V$, then~(\ref{eq:tube_mem_cond_beta}) holds if and only $\lambda_k^{(j,r)} \geq 0$ exists satisfying
\begin{equation}\label{eq:lambda_mode1_cond}
\begin{bmatrix}
\lambda_k^{(j,r)} V & 0 & ~~~(\Phi_k + C^{(j)}_k)^\top 
\\
\ast & \sigma^2 & {w^{(r)}}^\top
\\
\ast & \ast & V^{-1}
\end{bmatrix} \succeq 0 ,
\end{equation}
and ${\mu^{(j,q)}}^2 \geq  \lambda_k^{(j,r)}\beta_{k}^2 + \sigma^2$.
By Schur complements,~(\ref{eq:lambda_mode1_cond}) holds iff $\lambda_k^{(j,r)} V \succeq (\Phi_k + C^{(j)}_k)^\top 
\Psi^{(r)}(\Phi_k + C^{(j)}_k)$.
Choosing $\lambda_k$ as the smallest scalar satisfying $\lambda_k \geq \lambda_k^{(j,r)}$ therefore yields the sufficient conditions~{(\ref{eq:beta_mode1_dynamics}),
(\ref{eq:lambda_mode1_def})}.
\end{proof} 

\section{Ellipsoidal tube MPC subproblem}
At each iteration of the successive linearization algorithm,  a cost equivalent to an upper bound on the cost~(\ref{eq:cost}) is minimized
\[
J(\mathbf{x},\mathbf{u}) = \sum_{k=0}^{N-1} (\lVert x_k \rVert^2_Q +  \lVert u_k \rVert^2_R) + l_N^2 .
\]
Here $\mathbf{x}=\{x_0,\ldots,x_N\}$ is the state sequence generated by~(\ref{eq:system}) with control sequence $\mathbf{u}=\{u_0,\ldots,u_{N-1}\}$ and $l_N$ is a terminal cost discussed in Section~\ref{sec:termset}.
The minimization is performed subject to input and state constraints, terminal conditions, and constraints enforcing a sufficient cost decrease.
At each iteration $i$ at time $t$, we solve the following Second Order Cone Program (SOCP) given the current plant state $x^p_t$:
\begin{align} \label{opt:mpc}
& (\mathbf{v}^\star,\boldsymbol{\beta}^\star,\mathbf{z}^\star,\mathbf{l}^\star) =  \arg\min_{\mathbf{v},\boldsymbol{\beta},\mathbf{z},\mathbf{l}} \bar{J}^{(i)}_t =\sum_{k=0}^{N} l_k^2 \\ \nonumber
& \text{subject to, for $k=0,...,N-1$, and all $j\in\mathbb{N}_{\nu_{1}}$, $q\in\mathbb{N}_{\nu_\theta}$, 
}\\  \nonumber
& \quad \begin{aligned}
z_{k+1} & = \Phi_k z_k + B_kv_k \\
\beta_{k+1} & \geq 
(\lambda_k\beta_k^2 + \sigma^2)^\frac{1}{2} + 
\lVert C^{(j)}_k z_k + D^{(j)}_k v_k + \delta^{0\,(q)}_k \rVert_V \\
l_{k} & \geq \bigl( \lVert x^0_k +z_k\rVert_Q^2 
+ \lVert K(x^0_k + z_k) + v^0_k + v_k\rVert_R^2  \bigr)^{\frac{1}{2}}
\\
& \quad + \beta_k \lVert{V}^{-\frac{1}{2}}\rVert_{Q+K^\top R K}
\\
\mathcal{U} & \supset K\bigl(x_k^0+z_k+\mathcal{E}(V,\beta_k^2)\bigr) + v_k^0 + v_k \\
\mathcal{X} & \supset x_k^0+z_k+\mathcal{E}(V,\beta_k^2) \\
\mathcal{V} & \ni v_k^0 + v_k \\
\mathcal{S} & \supset z_k+\mathcal{E}(V,\beta_k^2) \\
\end{aligned} \\ \nonumber
& \text{and initial and terminal conditions} \\ \nonumber
& \quad \begin{aligned}
\beta_0 & \geq \lVert x^0_0 + z_0 - x^p_t\rVert_V \\
\Omega(x^0_N) & \ni (\|z_N\|_V, \beta_N) \\
l_N &\geq \hat{l}(z_N, \beta_N, x^0_N )
\end{aligned} \\ \nonumber
& \text{and, for iteration $i=1$,} \\ \nonumber
& \quad \bar{J}^{(i)}_t \leq \bar{J}^{(\mathit{final})}_{t-1} - \bigl(\lVert x_{t-1} \rVert^2_Q +  \lVert u_{t-1} \rVert^2_R -\hat{\sigma}^2 \bigr) \\ \nonumber
&\text{and, for iterations $i>1$,} \\  \nonumber
& \quad \bar{J}^{(i)}_t \leq \bar{J}^{(i-1)}_t .
\end{align}
Constraints of the form $\{x : H x \leq h\} \supset z + \mathcal{E}(V,\beta)$ are imposed via $[H]_i z + \beta \|\smash{V^{-\frac{1}{2}}}[H]_i^\top\| \leq [h]_i$ for each row $i$ of $H$.
The design of the terminal set $\Omega(x^0_N)$, terminal cost $\hat{l}$, and $\hat{\sigma}$ are discussed in Section~\ref{sec:termset}, and $\bar{J}_t^{(i-1)}$, $\bar{J}_{t-1}^{(\mathit{final})}$ denote the optimal objective at iteration $i-1$ and at the final iteration at time $t-1$.
The MPC strategy is summarised in Algorithm~\ref{alg:mpc}.



{\setlength{\algomargin}{0.83em}
\removelatexerror
\begin{algorithm2e}[h]
\DontPrintSemicolon
\SetKwInOut{Input}{Input}
\SetKwInOut{Output}{Output}
\Input{Initial perturbation sequence $\mathbf{v}^0$;
parameter estimate $\theta^0$; 
uncertainty bounds $\mathcal{W},\Theta$; 
cost weights $Q,R$;
tube parameters $\mathcal{S},\mathcal{V},V,\sigma$}
\Output{Control input $u_t$ at time steps $t = 0,1,\ldots$}
At time $t$, set $x^0_0 \gets x^p_t$, $\iter \gets 1$\;
\While{$\iter \leq \maxiter$ and $\|{\mathbf{v}^{\star}}\|\geq \algtol$}{
Compute $\mathbf{x}^0$ by simulating the nominal system 
(\ref{eq:nominal_system}) with initial state $x^0_0$ and perturbation sequence $\mathbf{v}^0$\;
Compute $\Phi_k,B_k$ in (\ref{eq:taylor}), bounds $\mathcal{W}^0_k, \mathcal{W}^1_k$, in (\ref{eq:W0_def}), (\ref{eq:W1_def}) and $\lambda_k$ in (\ref{eq:lambda_mode1_def}) for all $k\in \mathbb{N}_{[0,N-1]}$\; 
Attempt to solve Problem~{\rm(\ref{opt:mpc})} for $\mathbf{v}^{\star}$\; 
\If{Problem~{\rm(\ref{opt:mpc})} is infeasible}{
$\alpha \gets 1$, $\infeasflag \gets \algtrue$, $\lsiter \gets 1$\;
\While{$\lsiter \leq \maxlsiter$ and $\infeasflag = \algtrue$}{
$\alpha \gets \alpha/2$\;
\If{$\iter = 1$}{
$x^0_0 \gets x^0_{0,\old} + \alpha (x^0_0 -x^0_{0,\old})$\; 
}
$\mathbf{v}^0\gets \mathbf{v}^0_{\old} + \alpha (\mathbf{v}^0 -\mathbf{v}^0_{\old})$, $\lsiter \gets \lsiter + 1$\; 
Compute $\mathbf{x}^0$ using 
(\ref{eq:nominal_system}) with $x^0_0$ and $\mathbf{v}^0$\;
Compute $\Phi_k, B_k, \mathcal{W}^0_k, \mathcal{W}^1_k, \lambda_k, \forall k\in \mathbb{N}_{[0,N-1]}$\!\!\!\!\!\; 
Attempt to solve Problem~{\rm(\ref{opt:mpc})} for $\mathbf{v}^{\star}$\; 
\If{Problem~{\rm(\ref{opt:mpc})} is feasible}{
$\infeasflag \gets \algfalse$\;
}
}
\If{$\infeasflag = \algtrue$}{
$\mathbf{v}^{\star}\gets 0$, 
$\mathbf{v}^0 \gets \mathbf{v}^0_{\old}$, $i\gets \maxiter$\;
\If{$\iter = 1$}{
$\mathbf{x}^0 \gets \mathbf{x}^0_{\old}$\; 
}
}
}
Store $\mathbf{v}^0_{\old} \gets \mathbf{v}^0$\;
Update $\mathbf{v}^0 \gets \mathbf{v}^0 +\mathbf{v}^{\star}$, $\iter\gets \iter+1$\; 
}
Apply the control input $u_t= K x^p_t+ v^0_0$ and set
$\mathbf{v}^0 \gets \{v^0_1, \ldots, v^0_{N-1},0\}$
$\mathbf{v}^0_{\old} \gets \{v^0_{1,\old}, \ldots, v^0_{N-1,\old},0\}$
$\mathbf{x}^0_{\old} \gets \{x^0_{1},\ldots,x^0_{N},f_K(x_{N}^0,0,\theta^0)\}$\; 
\caption{Ellipsoidal Tube MPC}\label{alg:mpc}
\end{algorithm2e}}

The main iteration of Algorithm~\ref{alg:mpc} (lines 2-23) computes the Jacobian linearization of the plant model about the nominal trajectory generated by $\mathbf{v}^0$ (lines 3-4), and attempts to solve Problem~(\ref{opt:mpc}) (line 5).
Although the constraints of Problem~(\ref{opt:mpc}) ensure that 
the nominal trajectory $\mathbf{x}^0$ generated by~(\ref{eq:nominal_system}) satisfies $(x_k^0,u_k^0) \in (\mathcal{X},\mathcal{U})$ $\forall k\in \mathbb{N}_{[0,N-1]}$ and $x_{N}^0 \in \mathcal{X}_N$, the linearized dynamics determined in line 4 may not define a feasible set of constraints for Problem~(\ref{opt:mpc}). We therefore perform a backtracking line search in lines~8-21,  exploiting knowledge of a perturbation sequence $\mathbf{v}^0_{\mathrm{old}}$ and an initial nominal state $x^0_0$ such that Problem~(\ref{opt:mpc}) is feasible. As we show in Section~\ref{sec:stability}, this provides a guarantee of feasibility of at least one instance of Problem~(\ref{opt:mpc}) at each time step $t$.

\section{Terminal constraints and terminal cost} \label{sec:termset}
This section discusses how to compute the parameters $V,K,\sigma,\hat{\sigma}$ in problem~\eqref{opt:mpc}
and how to construct the terminal set $\Omega$ and terminal cost $\hat{l}(z,\beta,x^0)$.
In order to design a terminal cost and constraint set providing recursive feasibility and stability guarantees, we define the predicted trajectories of the model~(\ref{eq:system}) beyond the initial $N$-step prediction horizon by setting $v_k = 0$, $v^0_k = 0$ and $u_k=K x_k$ for  $k\geq N$.
To allow robust linear control design methods, we construct a linear difference inclusion (LDI) (e.g.~\cite{Boy94}) for each basis function $f_i$ in~($\ref{eq:theta_expansion}$) in a neighbourhood of $(x,u)=(0,0)$. Considering all affine combinations of the individual LDIs over the parameter set $\Theta_0$ yields an aggregate LDI. 
Hence, for all $\theta \in \Theta_0$, 
$x\in\hat{\mathcal{X}}$, $u\in\hat{\mathcal{U}}$,
for given polytopic sets $\hat{\mathcal{X}}$, $\hat{\mathcal{U}}$, let
\[
\begin{bmatrix}
\nabla_x f(x, u, \theta) & \nabla_u f(x,u,\theta)\end{bmatrix} \in \mathrm{co} \bigl\{ \begin{bmatrix}\hat{A}^{(j)}  & \hat{B}^{(j)} \end{bmatrix}, \, j \in \mathbb{N}_{\hat{\nu}} \bigr\} ,
\]
then, for all $(x_k,u_k)\in\hat{\mathcal{X}}\times\hat{\mathcal{U}}$ we have
\begin{equation}
f(x_k,u_k,\theta) \in\mathrm{co}\{\hat{A}^{(j)} x_k + \hat{B}^{(j)} u_k, \, j \in \mathbb{N}_{\hat{\nu}} \} .
\label{eq:ldi_term}
\end{equation} 
To simplify notation, let $\bar{\mathcal{X}} = \hat{\mathcal{X}}\cap\{x: Kx \in \mathcal{\hat{U}}\}$.

\begin{remark}
The LDI~(\ref{eq:ldi_term}) can be computed using the vertices of $\Theta_0$ and bounds on the Jacobians of the basis functions $f_i$ in (\ref{eq:theta_expansion}), analogously to the bounds on $\delta_k^1$ in Section~\ref{sec:error_bounds}.
\end{remark}

\begin{lemma} \label{lem:cost_bound}
For all $x\in\smash{\bar{\mathcal{X}}}$
the inequality
\begin{equation} \label{eq:cost_bound}
\lVert x \rVert^2_V - \lVert f(x,Kx,\theta) +w \rVert^2_V
\geq \lVert x \rVert^2_Q +  \lVert K x  \rVert^2_R - \sigma^2
\end{equation}
holds for all $\theta\in \Theta_0$ and all $w\in\mathcal{W}$ for positive definite $V$ and some scalar $\sigma$ if the following Linear Matrix Inequality (LMI) holds for all $j\in \mathbb{N}_{\hat{\nu}}$ and $ r \in \mathbb{N}_{\nu_w}$:
\begin{equation} \label{eq:cost_lmi}
\begin{bmatrix}
S & 0 & (\hat{A}^{(j)}S + \hat{B}^{(j)}Y)^\top & S & Y^\top \\
\ast & \tau & {w^{(r)}}^\top & 0 & 0 \\
\ast & \ast & S & 0 & 0 \\
\ast & \ast & \ast & Q^{-1} & 0 \\
\ast & \ast & \ast & \ast & \mathbb{R}^{-1}
\end{bmatrix} \succeq 0
\end{equation} with $V=S^{-1}$, $\sigma^2=\tau$, and $K =YV$.
\end{lemma}

\begin{proof}
This follows by substituting (\ref{eq:ldi_term}) into condition (\ref{eq:cost_bound})
and by considering Schur complements.
\end{proof}

To compute $V$, $\sigma$ and $K$ we minimize $\tau$ subject to (\ref{eq:cost_lmi}) by solving a semidefinite program. This approach is justified since (\ref{eq:cost_bound}) implies that $\sigma^2$ bounds the time-average value of the stage cost in (\ref{eq:cost}) as $t\to\infty$ under the control law $u_t=Kx_t$.

To determine the terminal constraint set for the tube parameters in Problem~(\ref{opt:mpc}) we first introduce the notation $\smash{\hat{Q}} = Q+K^\top R K$, $\smash{\hat{\Phi}^{(j)}} = \smash{\hat{A}^{(j)}+\hat{B}^{(j)}K}$, $j\in\mathbb{N}_{\hat{\nu}}$, and define
\begin{align}
\hat{\lambda} &= 1 - \sigma_{\min} ( V^{-\frac{1}{2}} \hat{Q} V^{-\frac{1}{2}} )
\label{eq:lambda_term}
\\
d_\Theta &= \max_{\theta^0,\theta^1 \in\Theta_0} \| \theta^0 - \theta^1 \|_1
\label{eq:Theta_diam}
\\
d_{\hat{\Phi}} &= \max_{j,k\in\mathbb{N}_{\hat{\nu}}} \|
\hat{\Phi}^{(j)} - \hat{\Phi}^{(k)}\|_V .
\label{eq:Phi_diam}
\end{align}%
Here $\hat{\lambda}\in [0,1)$ since a Schur complement of (\ref{eq:cost_lmi}) implies $V\succeq \smash{\hat{Q}}$.
We further assume $f_{K,i}(x,0)$ for each $i\in\{0,\ldots,p\}$ is $L$-Lipschitz continuous with respect to $x$, for all $x\in\smash{\bar{\mathcal{X}}}$. 

\begin{lemma}\label{lem:term_tube}
If $v_k=v^0_k=0$
in~(\ref{eq:taylor}) and (\ref{eq:decomp})-(\ref{eq:e_dynamics}), and if 
$x^0_k\in\smash{\bar{\mathcal{X}}}$ and $x_k \in x^0_k+z_k+\mathcal{E}(V,\beta_k^2)\subseteq\smash{\bar{\mathcal{X}}}$,
then $e_k \in \mathcal{E}(V,\beta_k^2)$ and
\begin{align}
\|x^0_{k+1}\|_V &\leq \hat{\lambda}^{\frac{1}{2}} \|x^0_k\|_V 
\label{eq:x0_term_bound}
\\
\|z_{k+1}\|_V &\leq \hat{\lambda}^{\frac{i}{2}} \|z_k\|_V 
\label{eq:z_term_bound}
\\
\beta_{k+1} &\geq (\hat{\lambda} \beta_{k}^2 +\sigma^2)^\frac{1}{2} + d_{\hat{\Phi}} \|z_{k}\|_V + 
d_\Theta L \|x_k^0\|_V .
\label{eq:beta_term_bound}
\end{align}
\end{lemma}

\begin{proof}
From (\ref{eq:ldi_term}) with $u_k=Kx_k$ and $x^0_k,z_k,e_k\in \bar{\mathcal{X}}$ we get
\begin{align*}
x^0_{k+1} &= f_K(x^0_k , 0, \theta^0)
\in\mathrm{co}\{\hat{\Phi}^{(j)} x^0_k , \, j \in \mathbb{N}_{\hat{\nu}} \} 
\\
z_{k+1} &= \Phi_k z_k 
\in\mathrm{co}\{\hat{\Phi}^{(j)} z_k , \, j \in \mathbb{N}_{\hat{\nu}} \} 
\\
e_{k+1} &= \Phi_k e_k + w_k + \delta^0_k + \delta^1_k 
\\
&\in\mathrm{co}\{\hat{\Phi}^{(j)} e_k , \, j \in \mathbb{N}_{\hat{\nu}} \} + w_k + \mathrm{co}\{ (\hat{\Phi}^{(j)} \!-\! \Phi_k) z_k \} + \delta^0_k 
\end{align*}
but (\ref{eq:cost_lmi}) ensures that, for all  $w\in\mathcal{W}$,  $j \in \mathbb{N}_{\hat{\nu}}$, and all $x\in\mathbb{R}^{n_x}$,
\begin{align*}
\|\hat{\Phi}^{(j)} x\|_V^2 
&\leq \|x\|_V^2 - \|x\|_{\hat{Q}}^2 \leq \hat{\lambda} \|x\|_V^2
\\
\|\hat{\Phi}^{(j)} x + w\|_V^2 
&\leq \|x\|_V^2 - \|x\|_{\hat{Q}}^2 + \sigma^2 \leq \hat{\lambda} \|x\|_V^2 + \sigma^2
\end{align*}
since (\ref{eq:lambda_term}) implies $V - \smash{\hat{Q}} \preceq \smash{\hat{\lambda}}V$. 
Furthermore, for all $j\in\mathbb{N}_{\hat{\nu}}$,
\[
\|\mathrm{co}\{ (\hat{\Phi}^{(j)} - \Phi_k) z_k \}\|_V 
\leq 
d_{\hat{\Phi}} \| z_k\|_V ,
\]
and $\delta^0_k = f_K(x_k^0,0,\theta - \theta^0)$ satisfies
\[
\|\delta^0_k\|_V \leq L \|x_k^0\|_V \|\theta \!-\! \theta^0\|_1
\leq 
d_{\Theta} L \| x^0_k\|_V .
\]
Hence $x^0_k$, $z_k$ satisfy (\ref{eq:x0_term_bound}), (\ref{eq:z_term_bound}), and $e_k$ satisfies  
\[
\|e_{k+1}\|_V \leq 
(\hat{\lambda} \|e_{k}\|_V^2 +\sigma^2)^\frac{1}{2}
+ d_{\hat{\Phi}} \|z_{k}\|_V + 
d_\Theta L \|x_k^0\|_V ,
\]
from which the bound~(\ref{eq:beta_term_bound}) follows.
\end{proof}

The terminal set $\Omega(x^0_N)$ is constructed so that 
$Hx_k \leq h$ for all $k\geq N$ whenever $(\|z_N\|_V,\beta_N) \in\Omega(x^0_N)$, where
\[
\{x : Hx \leq h\}
 = \mathcal{X} \cap \hat{\mathcal{X}} \cap \{x : Kx \in \mathcal{U} \cap \hat{\mathcal{U}}\} 
\]
is the aggregate constraint set.
To ensure this we impose the condition $\|x^0_k + z_k\|_V + \beta_k \leq \hat{\rho}$ for all $k\geq N$, where 
\begin{equation}\label{eq:rho_def}
\hat{\rho} = \min_i \bigl\{ [h]_i/\|[H]_i^\top\|_{V^{-1}}\bigr\} ,
\end{equation}
by defining the terminal constraint set for $(\|z_N\|_V, \beta_N)$ as%
\begin{align}
& \Omega(x^0_N) = \bigl\{
(r,\beta_N) : 
\beta_N \leq \hat{\rho} - (r + \|x^0_N\|_V)
\nonumber\\
& \quad \text{and } \exists \beta_{k} \text{ satisfying, for }
k=N+1,\ldots, N+\hat{N}, 
\nonumber\\
& \quad \beta_{k} \geq 
\smash{(\hat{\lambda}\beta_{k-1}^2 + \sigma^2)^{\frac{1}{2}}}
+ \smash{\hat{\lambda}^{\frac{(k-N-1)}{2}}} (r d_{\hat{\Phi}} +  d_\Theta L \|x^0_N\|_{V}) ,
\nonumber\\
& \quad \beta_{k}
\leq \hat{\rho} - \smash{\hat{\lambda}^\frac{k-N}{2}}(r + \|x^0_N\|_V) \bigr\}
\label{eq:term_constraint_set}
\end{align}
with $\hat{N}$ chosen large enough to satisfy 
\begin{align}
\max_{(r,\beta_N)\in\Omega(x^0_N)} & \bigl\{
(\smash{\hat{\lambda}}\beta_{N+\hat{N}\!}^2 + \sigma^2)^{\frac{1}{2}\!} 
+ \smash{\hat{\lambda}^\frac{\hat{N}}{2}\!} (r d_{\hat{\Phi}} +  d_\Theta L \|x^0_N\|_V) 
\nonumber \\
&\quad + \smash{\hat{\lambda}^\frac{\hat{N}+1}{2}} (r +  \|x_k^0\|_V \bigr\} 
\leq \hat{\rho}
.
\label{eq:Nc_def}
\end{align}

\begin{lemma}\label{lem:termset}
If $(\|z_N\|_V,\beta_N)\in\Omega(x^0_N)$, then $Hx_N \leq h$ 
and  
$(\|z_{k}\|_V,\beta_{k})\in \Omega(x^0_k)$
for all $k > N$, 
where $\beta_k$ for all $k>N$ satisfies (\ref{eq:beta_term_bound}) with the inequality replaced by equality.
\end{lemma}

\begin{proof}
Suppose $(\|z_N\|_V,\beta_N)\in\Omega(x^0_N)$, then 
$Hx_N\leq h$ 
for all $x_N\in x^0_N+z_N+\mathcal{E}(V,\beta_N^2)$ 
since ${\|x^0_N + z_N\|_V + \beta_N \leq \hat{\rho}}$. If $\beta_{N+1}$ is equal to the rhs of (\ref{eq:beta_term_bound}), then from (\ref{eq:x0_term_bound}), (\ref{eq:z_term_bound}), (\ref{eq:Nc_def}) we have $(\|z_{N+1}\|_V,\beta_{N+1})\in\Omega(x^0_{N+1})$. 
By induction this argument implies $(\|z_{k}\|_V,\beta_{k})\in \Omega(x^0_k)$ for all $k>N$.
\end{proof}

\begin{remark}
If $\sigma^2/(1-\hat{\lambda}) < \hat{\rho}^2$, 
then 
$\smash{\hat{N}}$ in (\ref{eq:Nc_def}) must be finite 
since
$\beta_{k} = \sigma/\smash{(1-\hat{\lambda})^{\frac{1}{2}}}$ is strictly feasible for~(\ref{eq:Nc_def}) in the limit as $\smash{k\to\infty}$.
The condition $(\|z_N\|_V,\beta_N)\in \Omega(x^0_N)$ introduces $2\smash{\hat{N}}$ second order cone constraints and $\smash{\hat{N}}$ additional scalar variables $\smash{\beta_{N+1},\ldots,\beta_{N+\hat{N}}}$ into Problem~(\ref{opt:mpc}). 
%
Using $(\hat{\lambda}\beta^2 + \sigma^2)^{\frac{1}{2}} \leq \hat{\lambda}^{\frac{1}{2}}\beta + \sigma$, we obtain a sufficient condition for (\ref{eq:Nc_def}):
\begin{align}
\max_{(r,\beta_N)\in\Omega(x^0_N)} & \bigl\{
\smash{\hat{\lambda}^\frac{1}{2}}\beta_{N+\hat{N}} + \sigma 
+ \smash{\hat{\lambda}^\frac{\hat{N}}{2}\!} (r d_{\hat{\Phi}} +  d_\Theta L \|x^0_N\|_V) 
\nonumber \\
&\quad + \smash{\hat{\lambda}^\frac{\hat{N}+1}{2}} (r +  \|x_k^0\|_V \bigr\} 
\leq \hat{\rho} ,
\label{eq:Nc_check}
\end{align}
which can be checked for given $\hat{N}$ by solving a SOCP. 
\end{remark}



The terminal cost $\hat{l}$ and parameter $\hat{\sigma}$ in Problem~(\ref{opt:mpc}) are constructed to ensure closed-loop stability via the condition
\begin{multline}\label{eq:term_cost}
\hat{l}^2(z_N,\beta_N,x^0_N) - \hat{l}^2(z_{N+1},\beta_{N+1},x^0_{N+1}) 
\\
\geq (\|x^0_N+z_N\|_{\hat{Q}} + \beta_N \|V^{-\frac{1}{2}}\|_{\hat{Q}})^2 - \hat{\sigma}^2 .
\end{multline}
In terms of the variables $\beta_{N},\ldots,\beta_{N+\hat{N}}$ in~(\ref{eq:term_constraint_set}), we define
\begin{align}
& 
\hat{l}^2(z_N,\beta_N,x^0_N) = \sum_{k=0}^{\smash{\hat{N}}} 
l_{N+k}^2
\label{eq:hat_l_def}
\\
& l_{N+k} = \begin{cases}
\hat{\lambda}^{\frac{k}{2}}(\|x^0_N\|_V + \|z_N\|_V) + \beta_{N+k} & 0\leq k <\hat{N}
\\
\gamma \hat{\lambda}^{\frac{\hat{N}}{2}}(\|x^0_N\|_V + \|z_N\|_V) + \gamma\beta_{N+\hat{N}} &
k = \hat{N}
\end{cases}
\nonumber
\\
& \hat{\sigma} = 
\gamma \sigma + 
\gamma  \hat{\lambda}^{\frac{\hat{N}}{2}} (d_{\hat{\Phi}} r_{\max} + d_\Theta L\|x^0_N\|_V)
\label{eq:hat_sigma_def}
\end{align}
with 
$\gamma^2 = 1/(1-\hat{\lambda}^{\frac{1}{2}})$
and
$r_{\max} = \max_{(r,\beta)\in\Omega(x^0_N)} r$.

\begin{lemma}\label{lem:term_cost}
  If $\hat{l}$ and $\hat{\sigma}$ are given by (\ref{eq:hat_l_def}), (\ref{eq:hat_sigma_def}), then~(\ref{eq:term_cost}) holds for all $(z_N,\beta_N)$
 satisfying $(\|z_N\|_V,\beta_N)\in\Omega(x^0_N)$.
\end{lemma}

\begin{proof}
This results from substituting (\ref{eq:hat_l_def})-(\ref{eq:hat_sigma_def}) into (\ref{eq:term_cost}), and using (\ref{eq:x0_term_bound})-(\ref{eq:beta_term_bound}) and the following inequality, which holds for any scalars $a,b,\hat{\lambda}>0$:
$(\smash{\hat{\lambda}^{\frac{1}{2}}}a + b)^2 \leq 
\hat{\lambda}^{\frac{1}{2}} a^2 + b^2/(1-\hat{\lambda}^{\frac{1}{2}})$.
\end{proof}%

A procedure for computing the parameters defining the terminal cost and constraint set is summarised in Algorithm~\ref{alg:term}.%

{\setlength{\algomargin}{1em}
\removelatexerror
\begin{algorithm2e}
\DontPrintSemicolon
\SetKwInOut{Input}{Input}
\SetKwInOut{Output}{Output}
\Input{Bounds $\hat{\mathcal{X}}, \hat{\mathcal{U}}, \Theta_0$, matrices $\hat{A}^{(j)},\hat{B}^{(j)}$ in~(\ref{eq:ldi_term}); 
disturbance, state and control sets $\mathcal{W}$, $\mathcal{X}$, $\mathcal{U}$; 
scalar $\hat{\rho}$ in~(\ref{eq:rho_def}); 
cost weights $Q$, $R$}
\Output{$V$, $\sigma$, $K$, $\hat{\lambda}$, $\gamma$, $\hat{N}$, $\hat{\sigma}$}
{
At time $t=0$: Solve $(S^\star,Y^\star, \tau^\star) = \arg\min_{S,Y,\tau} \tau$ s.t.~(\ref{eq:cost_lmi}) and set $V\gets (S^\star)^{-1}$, $K\gets Y^\star V$, $\sigma^2 \gets \tau^\star$, 
$\hat{\lambda} \gets 1 - \sigma_{\min} ( V^{-\frac{1}{2}} \hat{Q} V^{-\frac{1}{2}} )$,
$\gamma \gets 2\|V^{-\frac{1}{2}}\|_{\hat{Q}}^2/(1-\hat{\lambda}^{\frac{1}{2}\!})$\;
At times $t\geq 0$ in steps 4 and 14 of Alg.~\ref{alg:mpc}: 
Set $\hat{N}$ to a prior estimate (e.g.~$\hat{N}\gets 1$) and check (\ref{eq:Nc_check}); increase $\hat{N}$ until~(\ref{eq:Nc_check}) is satisfied; 
compute $\hat{\sigma}$ using (\ref{eq:hat_sigma_def})
}
\caption{Computation of terminal parameters}\label{alg:term}
\end{algorithm2e}}


\section{Recursive feasibility and stability}\label{sec:stability}

This section considers the closed-loop properties of the control law proposed in Algorithms~\ref{alg:mpc} and \ref{alg:term}. 
If the computation in Algorithm~\ref{alg:term} at time $t=0$ is feasible and if there exists a control perturbation sequence $\mathbf{v}^0$ such that Algorithm~\ref{alg:mpc} is feasible at $t=0$, then the system~($\ref{eq:system}$) with the control law of Algorithm~\ref{alg:mpc} can be shown to robustly satisfy the constraints $x_t\in\mathcal{X}$ and $u_t\in\mathcal{U}$ at all times $t \geq 0$. In addition, for the closed-loop system the asymptotic time-average of the stage cost  $\lVert x_t \rVert^2_Q +  \lVert u_t \rVert^2_R$ does not exceed the time-average of $\hat{\sigma}_t$.


We first show that Problem~(\ref{opt:mpc}) in Alg.~\ref{alg:mpc} is feasible at each time step using an inductive argument considering the feasibility of (\ref{opt:mpc}) in iteration $i+1$ at time $t$ assuming feasibility in iteration $i$ at time $t$, and the feasibility of (\ref{opt:mpc}) in the first iteration at time $t+1$ assuming feasibility in the final iteration of time $t$. In the latter case we use the following observation. 

\begin{lemma}\label{lem:beta_dynamics_term_bound}
If 
$x^0_k + \mathcal{S} \subseteq \hat{\mathcal{X}}$, $v^0_k=0$ and $Kx^0_k + K \mathcal{S} \subseteq \hat{\mathcal{U}}$, then 
\begin{multline}\label{eq:beta_dynamics_term_bound}
(\lambda_k \beta_k^2 + \sigma^2)^{\frac{1}{2}} + \max_{j\in\mathbb{N}_{\nu_1}, q\in\mathbb{N}_{\nu_\theta}} \lVert C^{(j)}_k z_k + \delta^{0\,(q)}_k \rVert_V\\
\leq
(\hat{\lambda}\beta_k^2 + \sigma^2)^{\frac{1}{2}} + d_{\hat{\Phi}} \|z_k\|_V + d_\Theta \| x^0_k \|_V
\end{multline}
for all $z_k\in\mathbb{R}^{n_x}$\!, $\beta_k\in\mathbb{R}$,  where $\lambda_k$, $\hat{\lambda}$ are given by (\ref{eq:lambda_mode1_def}), (\ref{eq:lambda_term}).
\end{lemma}

\begin{proof}
From (\ref{eq:lambda_term}) we have
$V - \smash{\frac{1}{\gamma}(Q+K^\top R K)} \preceq \smash{\hat{\lambda}}V$,
so~(\ref{eq:cost_lmi}) implies, for all $j\in \mathbb{N}_{\hat{\nu}}$ and $r \in \mathbb{N}_{\nu_w}$
\begin{equation}\label{eq:lambda_mode1_cond_term}
\begin{bmatrix}
\hat{\lambda} V & 0 & \hat{\Phi}{\mbox{}^{(j)}}^\top
\\
\ast & \sigma^2 & {w^{(r)}}^\top
\\
\ast & \ast & V^{-1}
\end{bmatrix} \succeq 0 .
\end{equation}
Furthermore, if $x^0_k + \mathcal{S} \subseteq \hat{\mathcal{X}}$, $v^0_k=0$ and ${Kx^0_k + K\mathcal{S}\subseteq \hat{\mathcal{U}}}$, then
$\mathrm{co}\{ \hat{\Phi}^{(j)}, \, j \in \mathbb{N}_{\hat{\nu}} \} \supseteq \mathrm{co}\{ \Phi_k + C_k^{(j)}, \, j\in\mathbb{N}_{\nu_1} \}$. Comparing  
(\ref{eq:lambda_mode1_cond_term}) and (\ref{eq:lambda_mode1_cond}) yields
$\hat{\lambda} \geq \lambda^{(j,r)}$ and 
$\hat{\lambda} \geq \lambda_k$. Since $d_{\hat{\Phi}} \|z_k\|_V + d_\Theta \|x^0_k\|_V \geq \max_{l, q} \lVert C^{(l)}_k z_k + \delta^{0\,(q)}_k \rVert_V$ we obtain (\ref{eq:lambda_mode1_cond_term}).
%
\end{proof}

\begin{theorem}\label{thm:feasibility}
If at time $t=0$, $\mathbf{v}^0$ and $x^0_0=x^p_0$ generate a nominal trajectory such that Problem~(\ref{opt:mpc}) is feasible, then 
in any iteration $i>1$ of Algorithm~\ref{alg:mpc} at $t\geq 0$, Problem~(\ref{opt:mpc}) is feasible with
$\mathbf{v}^0 = \mathbf{v}^0_{\mathrm{old}}$,
and in iteration $i=1$ at any time $t>0$, Problem~(\ref{opt:mpc}) is feasible if  $\mathbf{v}^0 = \mathbf{v}^0_{\mathrm{old}}$ and $x^0_0 = x^0_{0,\mathrm{old}}$.
\end{theorem}

\begin{proof}
This follows by induction from the following three cases. 
(i) If, in  iteration $\iter$ at time $t$, $\mathbf{v}^0$ and $x^0_0$ generate a nominal trajectory such that Problem~(\ref{opt:mpc}) in line~5 of Algorithm~\ref{alg:mpc} is feasible, then line~22 trivially ensures that~(\ref{opt:mpc}) is feasible in iteration $i+1$ time $t$ with $\mathbf{v}^0 = \mathbf{v}^0_{\mathrm{old}}$.
(ii) If, in the final iteration of time $t$, $\mathbf{v}^0$ and $x^0_0$ generate a nominal trajectory such that Problem~(\ref{opt:mpc}) in line~5 is feasible, then line~24 ensures 
(due to the definition of $\Omega$ and 
Lemmas~\ref{lem:term_tube}, \ref{lem:termset}, \ref{lem:term_cost}, and \ref{lem:beta_dynamics_term_bound})
that $\mathbf{v}^0=\mathbf{v}^0_{\mathrm{old}}$ and $x^0_0 = x^0_{0,\mathrm{old}}$ 
generate a nominal trajectory in iteration $i=1$ time $t+1$ such that Problem~(\ref{opt:mpc}) is feasible.
(iii) If, in any iteration $\iter$ and time $t$, Problem~(\ref{opt:mpc}) in line~5 is infeasible, then the feasibility of $\mathbf{v}^0 =\mathbf{v}^0_{\mathrm{old}}$ (or $\mathbf{v}^0 =\mathbf{v}^0_{\mathrm{old}}$ and $x^0_0=x^0_{0,\mathrm{old}}$ if $i=1$) implies that the line search in lines 8-21 necessarily terminates with $\mathbf{v}^0$ and $x^0_0$ that generate a nominal trajectory for which Problem~(\ref{opt:mpc}) is feasible, and hence feasibility of $\mathbf{v}^0 = \mathbf{v}^0_{\mathrm{old}}$ (or $\mathbf{v}^0 =\mathbf{v}^0_{\mathrm{old}}$ and $x^0_0=x^0_{0,\mathrm{old}}$) is ensured in iteration $i+1$ at time $t$ (or iteration $i=1$ at time $t+1$) by line~22 (or line~24, respectively).
\end{proof}

\begin{theorem}\label{thm:cost_bound}
If Problem~(\ref{opt:mpc}) is feasible at $t=0$, then the system~(\ref{eq:system}) with Algorithm~\ref{alg:mpc} satisfies $x_t\in \mathcal{X}$, $u_t\in\mathcal{U}$ and 
\begin{equation}\label{eq:closedloop_cost}
\limsup_{T\rightarrow \infty} \frac{1}{T}\sum_{t=0}^{T-1} (\left\lVert x_t \right\rVert^2_Q +  \left\lVert u_t \right\rVert^2_R)\leq \bar{\sigma}^2 
\end{equation}
where $\bar{\sigma} = \gamma \sigma + 
\gamma \hat{\rho} (d_{\hat{\Phi}} + d_\Theta L)$.

\end{theorem}

\begin{proof}
This follows from Theorem~\ref{thm:feasibility}, the constraints on $J_t^{(i)}$ in Problem~(\ref{opt:mpc}), and $\hat{\sigma} \leq \bar{\sigma}$, which implies, for all $t\geq 0$,
\begin{equation}\label{eq:cost_decrease}
\bar{J}^{(\mathit{final})}_{t} - \bar{J}^{(\mathit{final})}_{t+1} \geq \lVert x_{t} \rVert^2_Q +  \lVert u_{t} \rVert^2_R -\bar{\sigma}^2 ,
\end{equation}
Summing over $t$ yields~(\ref{eq:closedloop_cost}) since Theorem~\ref{thm:feasibility} and the boundedness of $\mathcal{X}$ and $\mathcal{U}$ imply
that $\smash{\bar{J}^{(\mathit{final})}_{t}}$ is bounded.
\end{proof}

\def\strut{\rule{0pt}{7pt}}
\begin{table*}[ht]
  \caption{Scaling of computation with problem dimensions}\label{table:computation}
\centerline{\begin{tabular}{@{}l | r r r r r r r r r r @{}}
$(n_x,n_u,n_{\theta})$\strut & $(2,1,2)$ & $(4,2,2)$ & $(4,2,4)$ & $(6,2,4)$ & $(5,2,5)$ & $(6,2,6) $ & $(8,2,8) $ & $(8,4,8) $ & $(10,4,10) $ & $(12,4,12)$ 
\\ \hline
Variables\strut & 48 & 60 & 60 & 62 & 61 & 62 & 64 & 84 & 86 & 88
\\
Equalities \strut & 22 & 44 & 44 & 66 & 55 & 66 & 88 & 88 & 110 & 132
\\
Inequalities\strut & 57 & 97 & 97 & 117 & 107 & 117 & 137 & 177 & 197 & 217
\\
SOC constraints\strut & 294 & 474 & 1274 & 1774 & 2184 & 3454 & 7314 & 7314 & 13334 & 21994
\\
Execution time (\SI{}{\second})\strut & 0.037 & 0.119 & 0.461 & 0.562 & 0.802 & 1.50 & 4.32& 4.23 & 13.17 & 50.18
\end{tabular}}
\end{table*}

\begin{remark}
  A corollary of Theorem~\ref{thm:cost_bound} is that~(\ref{eq:system}) is input-to-state stable (ISS)~\cite{Jiang01} under Algorithm~1. This follows from~(\ref{eq:cost_decrease}), which implies $\bar{J}^{(\mathit{final})}_{t}$ is an ISS-Lyapunov function since $\bar{\sigma}$ is a  $\mathcal{K}$-function of the uncertainty bounds $\max_r\|w^{(r)}\|, d_\Theta, d_{\hat{\Phi}}$.
  The latter follows from the definition of $\bar{\sigma}$ and from (\ref{eq:cost_lmi}), which implies that scaling the disturbance set ($\mathcal{W}\gets \kappa \mathcal{W}$ for $\kappa\in (0,1)$) scales the value of $\sigma$ by the same factor ($\sigma\gets \kappa \sigma$).
\end{remark}

\section{Numerical results} \label{simex}

The proposed controller was tested using randomly generated system models~(\ref{eq:system}) with quadratic nonlinearities:
\[
  f_0(x,u) = Ax + Bu, \ \ f_i(x,u) = e_i [x]_{j_i}^2, \ i\in \mathbb{N}_{n_\theta}
\]
where $A,B$ are randomly chosen matrices, $e_i$ is the $i$th column of the $n_x\times n_x$ identity matrix $I_{n_x}$, and $j_i$ is randomly chosen from $\mathbb{N}_{n_x}$ for each $i$ ($A, B$ and $j_1,\ldots,j_{n_\theta}$ are known to the controller). %
The disturbance set $\mathcal{W}$ belongs to a subspace of dimension $n_w\leq n_x$, and has vertices $w^{(r)} = B_w \hat{w}^{(r)}$, $r\in\mathbb{N}_{\nu_w}$, where $B_w\in\mathbb{R}^{n_x\times n_w}$ is a randomly generated full-column-rank matrix ($B_w$ and $\{\hat{w}^{(r)},\, r\in\mathbb{N}_{\nu_w}\}$ are known).

The true parameter $\theta^\ast$ is randomly chosen and  unknown to the controller, and the initial parameter set estimate is a random simplex $\Theta_0$ containing $\theta^\ast$ with $\max_{\theta\in\Theta_0}\| \theta - \theta^\ast\| \leq 0.05$.
For all $t$, $\Theta_t$ is updated using SME with estimation horizon $N_\Theta = 5$, and the nominal parameter vector $\theta^0_t$ is defined as the mean of the vertices of $\Theta_t$ (for details on SME see \cite{Bue24}).

The state and control sets are $\mathcal{X}= \{x: \| x\|_\infty\leq 10^6\}$, $\mathcal{U} = \{u: \| u \|_\infty \leq 1\}$, and the disturbance set is $\mathcal{W}=\{ B_w \hat{w} : \|\hat{w}\|_\infty \leq 0.01\}$. The cost matrices are $Q = I_{n_x}$, $R = I_{n_u}$, the prediction horizon is $N = 10$, and each simulation runs for $10$ time steps with a randomly chosen (feasible) initial condition.

The offline SDP in Algorithm~\ref{alg:term} is solved using a LDI model representation determined from the vertices of $\hat{\mathcal{X}} = \{x : \|x\|_\infty \leq 1.5 \}$ and $\Theta_0$. The state perturbation constraint set is a simplex: $\mathcal{S} = \{s : [-I_{n_x} \ {\bf 1}]^\top s \leq 0.5\}$, where ${\bf 1} = [1 \ \cdots \ 1]^\top$.
No control perturbation set is needed ($\mathcal{V} = \mathbb{R}^{n_u}$) because the system dynamics are linear in $u$. Since the nonlinear terms in the model are quadratic, the bounds on the error terms $\delta^0$ and $\delta^1$ are determined directly from the vertices of $\Theta_t$ and $\mathcal{S}$.


We apply Algorithm \ref{alg:mpc} with solution $\algtol = 10^{-3}$, using Gurobi~\cite{Gurobi} and Yalmip~\cite{Yalmip} to solve Problem~(\ref{opt:mpc})
as a Second Order Cone Program (Apple M3 Pro, 36 GB memory).

To investigate the computational requirements of the proposed algorithm we consider problems of varying sizes. Table~\ref{table:computation} shows how the numbers of variables and constraints and the time required to solve Problem~(\ref{opt:mpc}) vary with the state and control dimensions $n_x$, $n_u$ and with the dimension  $n_\theta$ of the unknown parameter vector. In each case the disturbance dimension is $n_w=2$, and the computation time in seconds is the mean of $10$ randomly generated problems of a given size.

The dominant factor determining the time needed to solve Problem~(\ref{opt:mpc}) is the number of SOC constraints. This is determined by the number of vertices of the uncertainty set $\mathcal{W}^1$ bounding $\delta^1$, which depends on the number ($n_\theta + 1$) of vertices of $\Theta_t$ and the number of vertices of the perturbation bound $\mathcal{S}$ that contribute to $f_i(x,u)$, $i\in\mathbb{N}_{n_\theta}$ (also $n_\theta + 1$). Hence the number of vertices of $\mathcal{W}^1$ is bounded by $(n_\theta + 1)^2$.
The times reported in Table~\ref{table:computation} show that computation time is proportional to $(n_\theta+1)^{4.2}$ (with $R^2$ value $0.97$).



\section{Conclusion}
This paper introduces a robust nonlinear MPC strategy with online model learning based on ellipsoidal tubes. The method results in a convex optimization problem, which is demonstrated to be efficiently solvable for a class of polynomial system models. The algorithm is recursively feasible and guarantees robust closed-loop stability. The approach scales favourably with the input and state dimensions of the model due to the formulation of Algorithm \ref{alg:mpc} in terms of ellipsoidal tubes.
Some promising future research directions are to use time-varying tube cross sections and local linear feedback gains computed online, and to consider convexification methods using differences of convex functions (as is done in~\cite{Bue24} with polytopic tubes) in the context of ellipsoidal tubes. 

\bibliographystyle{IEEEtran}
\bibliography{ellipsoidal}

\end{document}